\documentclass[12pt]{article}
\usepackage{amsmath,amsfonts,amssymb,amsthm,bm,verbatim}
\usepackage{fullpage}

\newtheorem{thm}{Theorem}
\newtheorem{lem}[thm]{Lemma}

\newtheorem{conj}[thm]{Conjecture}
\theoremstyle{definition}
\newtheorem{Def}[thm]{Definition}

\DeclareMathOperator\pr{Pr}

\title{Ramsey, Paper, Scissors}
\author{Jacob Fox\thanks{Department of Mathematics, Stanford University, Stanford, CA 94305, USA. Email: {\tt jacobfox@stanford.edu}. Research supported by a Packard Fellowship and by NSF Career Award DMS-1352121.}\and Xiaoyu He\thanks{Department of Mathematics, Stanford University, Stanford, CA 94305, USA. Email: {\tt alkjash@stanford.edu}. Research supported by a NSF GRFP grant number DGE-1656518.}\and Yuval Wigderson\thanks{Department of Mathematics, Stanford University, Stanford, CA 94305, USA. Email: {\tt yuvalwig@stanford.edu}. Research supported by a NSF GRFP grant number DGE-1656518.}}
\date{\today}

\begin{document}

\maketitle

\begin{abstract}
We introduce a graph Ramsey game called Ramsey, Paper, Scissors. This game has two players, Proposer and Decider. Starting from an empty graph on $n$ vertices, on each turn Proposer proposes a potential edge and Decider simultaneously decides (without knowing Proposer's choice) whether to add it to the graph. Proposer cannot propose an edge which would create a triangle in the graph. The game ends when Proposer has no legal moves remaining, and Proposer wins if the final graph has independence number at least $s$. We prove a threshold phenomenon exists for this game by exhibiting randomized strategies for both players that are optimal up to constants. Namely, there exist constants $0<A<B$ such that (under optimal play) Proposer wins with high probability if $s<A\sqrt{n}\log{n}$, while Decider wins with high probability if $s>B\sqrt{n}\log{n}$. This is a factor of $\Theta(\sqrt{\log{n}})$ larger than the lower bound coming from the off-diagonal Ramsey number $r(3,s)$.
\end{abstract}

\section{Introduction}
Ramsey's theorem states for every $m,s \ge 3$, there exists a least positive integer $r(m, s)$ for which every graph on $r(m, s)$ vertices has a clique of order $m$ or an independent set of order $s$. The study of the Ramsey numbers $r(m, s)$ and their many variations has a long history and holds a central place in extremal combinatorics.

One notable problem is the study of the asymptotic growth of $r(3,s)$. The classical argument of Erd\H os and Szekeres \cite{ErSz} yields an upper bound of $r(3,s) \leq \binom{s+1} {2}$, while an involved probabilistic construction of Erd\H os \cite{Er} shows that $r(3,s)=\Omega(s^2/\log^2 s)$. Spencer \cite{Sp1} later used the Lov\'asz Local Lemma to give a simpler proof of this same lower bound. The upper bound was then improved to $O(s^2/\log s)$ by Ajtai, Koml\'os, and Szemer\'edi \cite{AjKoSz} and later Shearer \cite{Sh1} improved the constant factor. Finally, the $\log s$ gap was closed by Kim \cite{Ki}, who proved that $r(3,s)=\Theta(s^2/\log s)$. Subsequently, Bohman \cite{Bohman} used the so-called triangle-free process to reprove Kim's lower bound, and further improvements \cite{BoKe,FiGrMo} to this analysis have determined the asymptotics of $r(3,s)$ up to a factor of $4+o(1)$. For more details on these  results, see the survey of Spencer \cite{Sp2} on the $r(3,s)$ problem.

Many interesting questions in graph Ramsey theory concern the game theory of various graph-building and graph-coloring games, usually played between two players. The earliest example of a graph Ramsey game was studied by Erd\H os and Selfridge~\cite{ErSe}, who studied a class of games known as positional games; a prototypical positional game is the Maker-Breaker game on graphs, in which two players, Maker and Breaker, take turns claiming the edges of a complete graph $K_n$, and Maker wins by building a graph with a prescribed property (such as containing a large clique). A more symmetric version of this game is usually just called the Ramsey game: given a fixed graph $H$, two players take turns claiming edges of a large complete graph until one player wins by building a copy of $H$. The mis\`ere version of this game, in which the first player to build a copy of $H$ loses, has also been studied. For an introduction to these topics, see the book of Hefetz, Krivelevich, Stojakovi{\'c}, and Szab{\'o} \cite{HKSS}.

A game that more closely resembles the one we study in this paper is the so-called online Ramsey game, which starts instead on a large empty graph and involves two players: Builder, who builds an edge on each turn, and Painter, who then paints it red or blue. Builder's goal is to build a monochromatic clique of a certain order $n$ using as few turns as possible. By Ramsey's theorem, Builder can always win by building the edges of a large complete graph, and Painter's goal is simply to delay this eventuality by as long as possible.

Ever since Erd\H os proved the lower bound $r(s, s) \ge 2^{s/2}$ on the classical Ramsey numbers using the probabilistic method, randomness has been ever-present in graph Ramsey theory. All known proofs of exponential lower bounds on $r(s, s)$ use the probabilistic method. Using a careful analysis of random play, Conlon, Fox, He, and Grinshpun~\cite{CoFoGrHe} recently proved that the online Ramsey game takes at least $2^{(2-\sqrt{2})s-O(1)}$ turns, making an exponential improvement on the trivial bound of $2^{s/2-1}$ in that setting. 

In certain cases, randomness is even baked directly into the definition of the game itself. Friedgut, Kohayakawa, R\"odl, Ruci\'nski, and Tetali~\cite{FKRRT} studied Ramsey games against a one-armed bandit: a variation of the online Ramsey game in which Builder builds a uniform random unbuilt edge on each turn from a fixed set of vertices and Painter must color these incoming edges red or blue while avoiding monochromatic triangles for as long as possible.

In this paper, we study another graph-building game we call Ramsey, Paper, Scissors, due to the simultaneous nature of the turns.
The game is played between two players, Proposer and Decider, on a fixed set of $n$ vertices, who jointly build a graph on these vertices one edge at a time. On each turn of the game, Proposer proposes a pair of vertices and Decider simultaneously decides whether to add this pair as an edge. Proposer cannot propose pairs that have been proposed before, nor pairs that would introduce a triangle to the graph if built. It is essential that Proposer and Decider make their moves simultaneously in each turn, so that Proposer doesn't know whether Decider intends to build the edge before proposing it, and Decider doesn't know which pair Proposer will propose. However, after both players make their choice, they each learn the other's choice (so Proposer learns whether the edge was added to the graph, and Decider learns which pair was proposed).

Ramsey, Paper, Scissors ends when Proposer has no legal moves, and Proposer wins if the independence number of the final graph is at least a predetermined value $s$. Note that if $n \geq r(3,s)$, then Proposer always wins, since the final graph must have independence number at least $s$. Thus, Proposer wins if $s \le c\sqrt{n \log n}$ for some constant $c > 0$. On the other hand, Decider can always win if $s=n$ by simply saying YES to the first pair Proposer proposes---this means that the final graph will have at least one edge, and thus must have independence number at most $n-1$. Finally, if $s$ is between these two bounds (i.e.\ $s<n<r(3,s)$), then both players can win with positive probability via randomized strategies. Namely, Proposer can propose pairs randomly, and if Proposer gets lucky, the only edges Decider will say YES to will be the $n-1$ edges incident to a single vertex, thus forcing the final graph to be the star $K_{1,n-1}$ with independence number $n-1$. Of course, if Decider says YES to fewer than $n-1$ pairs then if Proposer is lucky, the independence number is still at least $n-1$. 

For the other randomized strategy, Decider can fix a triangle-free graph $H$ on $n$ vertices with $m$ edges and independence number $s$, and say YES exactly $m$ times, determining these $m$ times completely at random. If Decider gets lucky, then exactly a copy of $H$ is built, thus achieving an independence number of $s$. However, both of these strategies will only succeed with extremely small probability, and it is therefore natural to ask for (randomized) strategies for both Proposer and Decider that succeed with high probability; as usual, we say that an event $E$ happens with high probability (w.h.p.)\ if $\pr(E) \to 1$ as $n \to \infty$. Indeed, our main theorem states that such strategies exist for both players for $s=\Theta(\sqrt{n} \log n)$.

\begin{thm} \label{thm:main}
If $s \leq \frac{1}{1000}\sqrt{n}\log n$ then Proposer can win w.h.p, while if $s \ge 1000 \sqrt{n} \log n$, Decider can win w.h.p.
\end{thm}
One surprising consequence of Theorem~\ref{thm:main} is that the fully random strategy is \emph{not} optimal for Proposer. Indeed, if Proposer were to play fully randomly (i.e.\ proposing a uniformly random pair among all remaining open pairs at each step), then Decider can choose to simply answer YES to every proposal. In that case, the game will simply follow the so-called triangle-free process, and Bohman~\cite{Bohman} proved that w.h.p.,\ the triangle-free process produces a graph with independence number $O(\sqrt{n \log n})$. Thus, if Proposer were to play fully randomly, Decider could respond with a strategy that saves a factor of $\Omega(\sqrt{\log n})$ from the optimum. Nevertheless, for Decider, a completely random strategy is optimal (up to a constant factor).

This paper is organized as follows. In the next section, we formally define the game and describe the Proposer and Decider strategies we will use to prove Theorem~\ref{thm:main}. In Section~\ref{sec:Decider-win}, we prove the upper bound in Theorem~\ref{thm:main} by studying the random Decider strategy. This argument is motivated by an old proof of Erd\H os~\cite{Er} that the off-diagonal Ramsey numbers satisfy $r(3,t) = \Omega(t^2/(\log t)^2)$, and relies on a concentration lemma of Conlon, Fox, He, and Grinshpun~\cite{CoFoGrHe}. This kind of argument has been generalized to prove lower bounds on all off-diagonal Ramsey numbers by Krivelevich~\cite{Kr}.

Finally, in Section~\ref{sec:Proposer-win}, we prove the lower bound in Theorem~\ref{thm:main}. This last argument is the most delicate, and requires two main ingredients: an appropriate ``semi-random'' strategy for Proposer and the analysis of an auxiliary game, whose upshot is that Azuma's inequality remains close to true, even when an adversary is allowed to weakly affect the outcomes taken by a sequence of random variables. 

In the concluding remarks, we mention some open problems and conjectures relating to this and other Ramsey games.

All logarithms are base 2 unless otherwise stated. For the sake of clarity of presentation, we systematically omit floor and ceiling signs whenever they are not crucial.

\section{Background and Winning Strategies}
Ramsey, Paper, Scissors is the following two-player graph-building game played on a fixed set $V$ of $n$ vertices. Each turn of the game yields a new graph $G_i=(V,E_i)$ on these vertices. The game is initialized with $G_0$ as the empty graph, and each $G_i$ will either be equal to $G_{i-1}$ or will add a single new edge to $G_{i-1}$. Throughout the game $G_i$ is required to be triangle-free. In order to ensure this, call a pair $\{x,y\} \in \binom V2 \setminus E_i$ \emph{closed} in $G_i$ if there is some $z \in V$ such that $\{x,z\}, \{y,z\} \in E_i$; this means that the pair $\{x,y\}$ cannot be added as an edge without introducing a triangle to the graph. Let $C_i$ denote the set of closed pairs in $G_i$. Call the pair $\{x,y\} \in \binom V2 \setminus E_i$ \emph{open} in $G_i$ if it is not closed in $G_i$, and let $O_i$ be the set of all open pairs in $G_i$. Thus, $\binom V2=E_i \sqcup O_i \sqcup C_i$. 

The game proceeds as follows. Initialize a ``forbidden'' set $F_0=\emptyset$; in general, $F_i$ will contain all pairs that Proposer has proposed up to turn $i$. On turn $i$, Proposer chooses a pair $\{x,y\} \in O_i \setminus F_i$. Decider \emph{does not} learn what pair Proposer has chosen, but must decide to answer either YES or NO. If Decider answers YES, then we add the edge $\{x,y\}$ to $G_{i}$, so that $E_{i+1}=E_{i} \cup \{xy\}$. Otherwise, if Decider answers NO, then $E_{i+1}=E_{i}$.  Finally, regardless of Decider's answer, the pair $\{x,y\}$ is added to the forbidden set, namely $F_{i+1}=F_i \cup \{xy\}$. Thus, Proposer can never propose the same pair twice. Both players can see the contents of $F_i$, so at this point Decider learns what pair was proposed. The game ends when Proposer has no legal moves left, on the turn $i$ where $O_i \subseteq F_i$. Proposer wins if the final graph thus produced contains an independent set of order $s$ (for some parameter $s$), and Decider wins otherwise.

We now sketch the strategies for Proposer and Decider which we will use to prove the two parts of Theorem~\ref{thm:main}. 
Decider's strategy is simply to say YES on each turn with probability $p=Cn^{-1/2}$ randomly and independently, for an appropriately chosen constant $C>0$. As a result, the final graph built will be a subgraph of the Erd\H{o}s--R\'{e}nyi random graph $G(n,p)$, where some edges may have been removed to make it triangle-free. We will show in Section~\ref{sec:Decider-win} that not too many such edges are removed from any set of a certain size, so the independence number at the end of the game will still be $\Theta(\sqrt{n}\log n)$.

As mentioned above, Proposer's strategy cannot be purely random, for otherwise Decider could respond by saying YES on every single turn. The resulting graph process will be exactly the triangle-free process, which would result in a final graph with an independence number of $\Theta(\sqrt{n \log n})$ w.h.p.\ \cite{Bohman}.

To achieve an independence number of $\Theta(\sqrt{n} \log n)$, Proposer will divide the game into $m$ stages (which we call epochs), where $m$ is extremely slowly growing (its growth is of order $\log ^* n$), and Proposer's strategy is to force Decider to progressively ratchet up the fraction of YES answers in each epoch. Roughly, this is done as follows, though the precise strategy is slightly more involved. Proposer divides the vertex set into $2m$ sets $U_i, V_i$, where $1\le i\le m$ of sizes approximately $|U_i|=|V_i| \approx n/2^{i+1}$. In epoch $i$, Proposer proposes all the pairs from $U_i$ to $V_j$ for $j\ge i$ and all the open pairs within $V_i$, mixing together these two kinds of pairs in a uniformly random order.

To see why Decider must answer YES more often with each epoch, note that in order to keep $V_i$ from containing a large independent set, Decider must answer YES with a certain density $p$. But then the number of edges built from $U_i$ to $V_{i+1}$ will have roughly the same density $p$, which will then close a large fraction of the pairs in $V_{i+1}$. Thus, on the next epoch, Decider must answer YES much more frequently in order to reach the same final edge density in $V_{i+1}$. After enough epochs, this becomes impossible. Proposer's strategy is slightly different from the one sketched above to simplify the analysis; it is formally laid out in Section \ref{sec:Proposer-win}, as are the details of this heuristic argument.

\section{Decider's winning strategy}\label{sec:Decider-win}
In this section, we show that Decider's random strategy wins w.h.p.\ when Proposer must build an independent set of order at least $1000\sqrt{n} \log n$.

\begin{thm}\label{thm:Decider-win}
 Decider has a randomized strategy such that w.h.p., no matter how Proposer plays, the final graph produced will contain no independent set of order at least $1000\sqrt{n}\log n$.
\end{thm}

We will require the following structural result about the Erd\H{o}s--R\'enyi random graph (see \cite[Lemma 18]{CoFoGrHe}). It shows that for suitable $p \approx n^{-1/2}$, when edges are removed from $G(n,p)$ until it is triangle-free, not many edges need to be removed from any given small subset. 

\begin{lem} \label{lem:common-neighbors}
Suppose $t$ is sufficiently large, $p = 20 (\log t) / t$, $n = 10^{-6} t^2 /(\log t)^2$,
and $G \sim G(n,p)$ is an Erd\H{o}s-R\'{e}nyi random graph. Then, w.h.p.\ there does not exist a set $S\subset V(G)$ of order $t$ such that more than $\frac{t^2}{10}$ pairs of vertices in $S$ have a common neighbor outside $S$. 
\end{lem}

In order to apply this lemma, we will need to understand the space of Proposer's strategies when Decider plays randomly.

\begin{Def}
Let $G$ be an arbitrary graph. Define a spanning subgraph $H\subset G$ to be a {\it reachable subgraph} of $G$ if for every edge $(u,v)\in E(G)\setminus E(H)$ there exists $w\in V(G)$ for which $(u,w),(v,w)\in E(H)$. 
\end{Def}

That is, every edge of $G$ not in $H$ is the third edge of a triangle with the other two edges in $H$.

\begin{proof}[Proof of Theorem~\ref{thm:Decider-win}]
Suppose $n,t,p$ satisfy the conditions of Lemma~\ref{lem:common-neighbors}. Decider's strategy is simply to answer YES with probability $p$ on every turn. Unless a pair $\{x,y\}$ is closed, Proposer must choose it eventually, in which case it is added to the graph with probability $p$. Thus, we may as well pretend that Decider samples a random graph $G\sim G(n,p)$ at the beginning of the game and answers YES to a pair $\{x,y\}$ if and only if $x\sim y$ in $G$. The final graph produced is thus some subgraph $G'\subseteq G$ of the random graph $G$, obtained by removing certain edges from triangles. In fact, it is easily seen that the final graph $G'$ must be a reachable subgraph of $G$.

We will show that w.h.p., every reachable subgraph of $G\sim G(n,p)$ has independence number less than $t$. Since $t \le 1000 \sqrt{n}\log n$, this completes the proof.

For each $S\subset V(G)$ of order $t$, define $X(S)$ to be the event that not more than $t^2/10$ pairs of distinct vertices in $S$ have a common neighbor outside $S$. By Lemma~\ref{lem:common-neighbors}, w.h.p.~all the events $X(S)$ occur.  

Let $I(S)$ be the event that there exists a reachable subgraph $H\subseteq G$ in which $S$ is an independent set. Conditioning on $X(S)$, at most $t^2/10$ pairs of vertices in $S$ have common neighbors outside $S$. Note that if $G[S]$ contains an edge $(u,v)$ with no common neighbor outside $S$, then $S$ is not an independent set in any reachable subgraph $H\subset G$. This is because either $(u,v)\in E(H)$ or else the two edges that form a triangle with $(u,v)$ must both lie in $E(H)$.
In either case there is an edge in $H[S]$.

Thus, since there are at least $\binom{t}{2}-\frac{t^2}{10}$ pairs in $S$ chosen by Proposer,
\[
\pr[I(S) | X(S)] \le (1-p)^{\binom{t}{2}-\frac{t^2}{10}} \le (1-p)^{t^2/4}.
\]
We now compute
\[
\pr\left[\bigvee_S I(S)\right] \le \pr\left[\overline{\bigwedge_S X(S)}\right] + \pr\left[\bigwedge_S X(S) \wedge \bigvee_S I(S)\right].
\]

The first summand goes to zero by Lemma~\ref{lem:common-neighbors}. The second can be bounded by a union of events of low probability, as follows.
\begin{align*}
\pr\left[\bigwedge_S X(S) \wedge \bigvee_S I(S)\right]
 & \le \sum_S \pr\left[\bigwedge_{S'} X(S') \wedge I(S)\right] \\
 & \le \sum_S \pr[X(S)\wedge I(S)] \\
 & \le \sum_S \pr[I(S) | X(S)] \\
 & \le \binom{n}{t} (1-p)^{t^2/4} \\
 & \le e^{t\log n} \cdot e^{-pt^2/4} \\
 & \le e^{2t\log t} \cdot e^{-5t \log t}.
\end{align*}

It follows that w.h.p., none of the events $I(S)$ occur. That is, there is no reachable subgraph of $G\sim G(n,p)$, and Proposer cannot win.
\end{proof}

\section{Proposer's winning strategy}\label{sec:Proposer-win}

In the previous section, we showed that Decider can win w.h.p.\ when $s = \Omega(\sqrt n \log n)$. In this section, we show that Proposer can match this bound (up to the constant factor), by exhibiting a strategy that yields an independence number of $\Omega(\sqrt{n}\log n)$ w.h.p.

\begin{thm}\label{thm:Proposer-win-adaptive}
Proposer has a randomized strategy such that w.h.p., no matter how Decider plays, the final graph produced will contain an independent set of order at least $\frac{1}{1000} \sqrt n \log n$. 
\end{thm}

We now formally describe Proposer's strategy. In the next subsection, we introduce an auxiliary game and its analysis, which will be crucial to the analysis of this strategy, and then we prove that this strategy indeed wins w.h.p.\ in Section \ref{subsec:strategy-works}.

Proposer begins by partitioning the vertex set as $U \sqcup V \sqcup A \sqcup B$ with $|U|=|V|=n/6$ and $|A|=|B|=n/3$, and labeling their vertices $u_1,\ldots,u_{n/6},v_1,\ldots,v_{n/6},a_1,\ldots,a_{n/3},b_1,\ldots,b_{n/3}$, respectively (we can assume for simplicity that $n$ is divisible by $6$). Everything significant Proposer does will be between $U$ and $V$, while $A$ and $B$ will only be used for ``clean-up'' to simplify the analysis. 

Proposer divides the game into $n/6$ \emph{periods.} At the beginning of period $i$, Proposer first compiles a list $L_i$ of open pairs to propose in this period, and then orders $L_i$ uniformly at random. During the period, Proposer proposes the pairs in $L_i$ one at a time in this order. By the choice of $L_i$, regardless of how Decider plays during period $i$, none of the pairs in $L_i$ will be closed, so Proposer will be able to propose all of them regardless of the chosen order or of Decider's choices. 

Proposer's list $L_i$ consists of all the pairs $\{u_i,v_j\}$ for $j>i$, together with all the pairs $\{v_i,v_j\}$ with $j>i$ that are open at the beginning of period $i$. Additionally, Proposer ensures that each period has length $n/3$ by adding some number of pairs $\{a_\ell,b_{\ell'}\}$ to $L_i$ until $|L_i| = n/3$. The list $L_i$ is designed so that \emph{all} pairs in $L_i$ are open at the start of period $i$ and remain open until they are proposed. Moreover, since the induced subgraph on $A\cup B$ stays bipartite throughout this whole process, all of the ``clean-up'' pairs in $L_i$ will remain open as well, and since we made $A$ and $B$ sufficiently large, Proposer will always have enough pairs to propose to ensure that $|L_i| = n/3$ for all periods. Finally, as stated above, once Proposer has compiled $L_i$, it is ordered uniformly at random, and Proposer proposes the pairs in $L_i$ according to that order.

Once Proposer has done this for periods $1,2,\ldots,n/6$, there will be many pairs that are still open and that Proposer has not yet proposed. Proposer will propose these in an arbitrary order; no matter what happens at this ``endgame'' stage, Proposer will have already guaranteed a sufficiently large independent set inside $V$, which will remain independent throughout the remainder of the game. 

For the analysis, we will also want to group periods into \emph{epochs}. To do this, we will pick a sequence of decreasing positive constants $\varepsilon_1,\varepsilon_2,\ldots,\varepsilon_m$ summing to $1/6$, and declare the first epoch to consist of the first $\varepsilon_1 n$ periods, the second epoch to consist of the subsequent $\varepsilon_2 n$ periods, and so on. The number of epochs, $m$, will be a very slowly growing function of $n$. Let $I_k \subseteq \{1,\ldots,n/6\}$ denote the set of periods comprising epoch $k$, and let
\[
	U_k=\{u_i\}_{i \in I_k} \qquad \qquad V_k=\{v_i\}_{i \in I_k}
\]
denote the set of vertices in $U$ and $V$ that are the roots of those pairs considered in epoch $k$. 

In order to prove that this strategy works w.h.p.,\ we will need a probabilistic tool which we call the \textit{bucket lemma,} and which is stated and proved in the next section. It describes the behavior of a different game, which is designed to model one period of Ramsey, Paper, Scissors from Decider's perspective. The bucket lemma can be thought of as an adaptive concentration inequality akin to Azuma's inequality; it says that even if an adversary has certain weak control over a sequence of random variables, he can't make it stray too far from its mean w.h.p.

\subsection{Coins and Buckets}
We begin with a straightforward tight concentration lemma. To prove it, we will also need the following concentration lemma of Bohman \cite{Bohman}, which is a generalization of Azuma's inequality; it says that if we have a martingale whose differences are bounded by different amounts from above and below, then we get concentration of the same order as a martingale with differences bounded by the geometric mean of the bounds.

\begin{lem}[{\cite[Lemmas 6 \& 7]{Bohman}}]\label{lem:Bohman-concentration}
    Suppose $0=Z_0,Z_1,\ldots,Z_m$ is a martingale such that for all $i$,
    \[
        Z_i -c_1 \leq Z_{i+1} \leq Z_i+c_2,
    \]
    where $0<c_1 \leq c_2/10$ are constants. Let $0<\lambda<mc_1$. Then
    \[
        \pr(|Z_m| \geq \lambda) \leq 2e^{-\frac{\lambda^2}{3 c_1 c_2 m}}.
    \]
\end{lem} 

Given integers $a,b>0$, define $X_{a,b}$ to be the following random variable with mean zero:
\begin{align*}
\Pr\left[X_{a,b} = \frac{1}{a}\right] & = \frac{a}{a+b} \\
\Pr\left[X_{a,b} = -\frac{1}{b}\right] & = \frac{b}{a+b}.
\end{align*}

\begin{lem}\label{lem:simple-concentration}
Suppose $a,b,\nu_0, \nu$ are positive integers with $a\le b$, $\nu = a + b$, and $\nu_0 \le \nu$, and $X_1, \ldots, X_\nu$ are $\nu$ independent random variables identically distributed as $X_{a,b}$. Let $S_m = \sum_{i\le m} X_i$. Then, for all $0 < t < \frac{a}{b}$,
\[
\Pr\left[\exists m \in [\nu_0, \nu]  \text{ \em such that } |S_m| \ge \frac{tm}{a}\right] \le \frac{40}{t^2} e^{-\frac{\nu_0 \nu t^2}{20a}}.
\]
\end{lem}
\begin{proof}
Fix an $m \ge \nu_0$, and notice that $S_0, S_1, \ldots, S_m$ is a martingale. The martingale $\{S_j\}$ is $(1/a)$-Lipschitz, so by Azuma's inequality, for all $t > 0$,
\begin{equation}\label{eq:azuma}
\Pr\left[|S_m| \ge \frac{tm}{a}\right] \le 2 e^{-m t^2}.
\end{equation}

We will use this bound when $a \geq b/10$. In the case that $a<b/10$, we use instead Lemma~\ref{lem:Bohman-concentration} with $\lambda = tm/a$, $c_1 = 1/b$, and $c_2 = 1/a$. The condition $\lambda < mc_1$ holds because we assumed $t< a/b$, so we have
\begin{equation}\label{eq:bohman-conc}
\Pr\left[|S_m| \ge \frac{tm}{a}\right] \le 2 e^{-\frac{mbt^2}{3a}}.
\end{equation}
Together with the fact that $b \ge \nu/2$, inequalities (\ref{eq:azuma}) and (\ref{eq:bohman-conc}) show that
\begin{equation}\label{eq:overall-conc}
\Pr\left[|S_m| \ge \frac{tm}{a}\right] \le 2 e^{-\frac{m\nu t^2}{20a}}
\end{equation}
for all $a\le b$.

In particular, applying the union bound over all $m \ge \nu_0$ to (\ref{eq:overall-conc}), we have that
\[
\Pr\left[\exists m \in [\nu_0, \nu]  \text{ such that\ } |S_m| \ge \frac{tm}{a}\right] \le \sum_{m \ge \nu_0} 2e^{-\frac{m\nu t^2}{20a}} \le \frac{2 e^{-\frac{\nu_0 \nu t^2}{20a}}}{1-e^{-\frac{\nu t^2}{20a}}}.
\]
Because $e^{-x} \le 1-x/2$ for $x \in [0,1]$, and because
\[
    \frac{\nu t^2}{20a} \leq \frac{(2b) (a/b)^2}{20a}=\frac{a}{10b}<1,
\]
we can bound $1-e^{-\frac{\nu t^2}{20a}} \ge \nu t^2/40a\ge t^2/20$ for all $0<t<a/b$. Thus,
\[
\Pr\left[\exists m \in [\nu_0, \nu]  \text{ such that\ } |S_m| \ge \frac{tm}{a}\right] \le \frac{40}{t^2} e^{-\frac{\nu_0 \nu t^2}{20a}},
\]
as desired.
\end{proof}

Consider the following game, which we call Coins and Buckets, and which we will use to model a single period of Ramsey, Paper, Scissors in which Proposer plays randomly. The game begins with two buckets $A,B$ of sizes $a,b$ respectively, and a total of $\nu=a+b$ coins to divide among them. Before the game starts, Proposer picks a random set $I$ from $\binom{[\nu]}{a}$ (but does not reveal the choice of the set to Decider). On step $i$ of the game, a coin is placed into one of two buckets: bucket $A$ if $i\in I$, and bucket $B$ otherwise. Before the coin is placed, Decider decides whether it is placed inside heads or tails, but Decider does not find out which bucket the coin enters until after the choice. Decider's goal is to make the distribution of heads in the buckets as uneven as possible. Namely, if we let $h$ be the total number of coins placed heads-up, and $h_A$ the number of coins in bucket $A$ that are heads-up, Decider wishes to maximize the error parameter $|h_A - ah/\nu|$, which we call his score.

We also pick a threshold $\nu_0$ such that we consider the game a forfeit if Decider picks heads fewer than $\nu_0$ times---that is, Decider receives a score of zero. Call this game the Coins and Buckets game with parameters $(a,\nu,\nu_0)$.

Under these conditions, we show that in Coins and Buckets, regardless of how Decider plays, w.h.p.\ the density of heads in the first bucket is not far from the overall density of heads.

\begin{lem}\label{lem:bucket}
Suppose Decider plays a game of Coins and Buckets with parameters $(a,\nu,\nu_0)$ where $a\le \nu/2$. If $h_A$ is the number of heads in the first bucket at the end of the game and $h$ is the total number of heads, then for any $0<t<a/\nu$,
\[
\Pr[(h \ge \nu_0) \wedge (|h_A - ah/\nu| \ge th)] \le \frac{80\sqrt{\nu}}{t^2}e^{-\frac{\nu_0\nu t^2}{20a}}.
\]
\begin{proof}
Of course we may assume that Decider chooses heads at least $\nu_0$ times. 

Let $b=\nu-a$. Instead of playing under the assumption that Proposer picks a fixed sequence out of $\binom{[\nu]}{a}$, we will simplify our analysis by considering the modification in which each coin is placed into the first bucket with probability $\frac{a}{\nu}$ independently; our first task is to show that this simplification is allowable, namely that a strong bound for this simplified model implies the desired bound in the original setting. To that end, let $E$ be the event that exactly $a$ of the coins fall in the first bucket in this new setting. 

The factorial function satisfies
\[
\sqrt{2\pi} \cdot \frac{\nu^{\nu+\frac{1}{2}}}{e^\nu} \le \nu! \le e\cdot \frac{\nu^{\nu+\frac{1}{2}}}{e^\nu}
\]
for all $\nu\ge 1$. Thus, we find that
\[
\Pr[E] = \binom{\nu}{a} \cdot\frac{a^ab^b}{\nu^\nu} \ge \frac{\sqrt{2\pi \nu}}{e^2\sqrt{ab}} \ge \frac{1}{2\sqrt{\nu}},
\]
where in the last step we used the AM-GM inequality in the form $\sqrt{ab} \le \nu/2$.

The original formulation of Coins and Buckets can be obtained by the modified version by conditioning on event $E$, and the probability of $E$ is at least $1/2\sqrt{\nu}$. Thus, it will suffice to show that in the modified setting,
\begin{equation}\label{eq:modified-goal}
\Pr\left[\left|h_A - \frac{a}{\nu}h\right| \ge th\right] \le \frac{40}{t^2}e^{-\frac{\nu_0 \nu t^2}{20a}},
\end{equation}
because conditioning on $E$ can multiply the failure probability by at most a factor of $\Pr[E]^{-1}$ (by Bayes' Theorem). We turn to proving (\ref{eq:modified-goal}) now.

Let $h_A^{(i)}$ and $h_B^{(i)}$ be the number of heads in buckets $A$ and $B$ respectively after $i$ turns of the game. Consider the sequence of random variables
\[
Z_i = \frac{1}{a}h_A^{(i)} - \frac{1}{b} h_B^{(i)}.
\]

On each of Decider's turns, there are two possibilities. If Decider picks tails, then $Z_i$ remains unchanged. On the other hand, if Decider picks heads, the coin goes into bucket $A$ with probability $a/(a+b)$, and bucket $B$ otherwise. Therefore, in this case, $Z_{i+1}-Z_i$ is distributed as $X_{a,b}$, the random variable defined for Lemma~\ref{lem:simple-concentration}.

We are ready to apply Lemma~\ref{lem:simple-concentration}. To see why, imagine that we ``pre-process'' some of the randomness in the game, as follows. Before the start of the game, we sample $\nu$ independent random variables $X_1, \ldots, X_\nu$ identically distributed as $X_{a,b}$. Keep a count $\ell$ of the total number of heads so far. On each turn that Decider picks tails, we place the coin into a bucket at random as before. However, on the turn Decider picks heads for the $\ell$th time, we place the coin into bucket $A$ if and only if $X_\ell = 1/a$. From Decider's perspective, this is the same game, since he doesn't know that the randomness was ``pre-processed''.

Under these assumptions, $Z_i$ will be equal to exactly $S_{h^{(i)}} = \sum_{j\le h^{(i)}} X_j$ where $h^{(i)}$ is the number of heads placed up through turn $i$. It follows that $Z_n=S_h$, where $h=h^{(\nu)}$ is the total number of heads placed. But then we can apply Lemma~\ref{lem:simple-concentration} to see that for $0<t<a/\nu<a/b$,
\begin{equation}\label{eq:random-heads}
\Pr[|Z_\nu| \ge th/a] \le \Pr\left[\exists m \in [\nu_0, \nu]  \text{ such that } |S_m| \ge \frac{tm}{a}\right] \le \frac{40}{t^2}e^{-\frac{\nu_0 \nu t^2}{20a}}.
\end{equation}
We now finish by noting that 
\[
|Z_\nu| = \left|\frac{1}{a} h_A - \frac{1}{b}h_B\right| \ge \left|\frac{1}{a} h_A - \frac{1}{\nu} h\right|,
\]
since $\frac 1 \nu h$ is a convex combination of $\frac 1a h_A$ and $\frac 1b h_B$, and so (\ref{eq:random-heads}) implies (\ref{eq:modified-goal}), as desired.
\end{proof}
\end{lem}

\subsection{Proposer's strategy works}\label{subsec:strategy-works}

We are ready to show that the strategy described at the beginning of  Section~\ref{sec:Proposer-win} indeed allows Proposer to produce an independent set of order $\frac{1}{1000}\sqrt{n}\log n$ w.h.p.

Recall the setup: $G_i=(V,E_i)$ is the graph built on turn $i$. The vertices are labeled $u_i,v_i, a_\ell, b_{\ell}$ where $1\le i\le n/6$ and $1\le \ell \le n/3$. Period $i$ consists of a total of $n/3$ turns, during which Proposer proposes only pairs of the forms $\{u_i, v_j\}$, $\{v_i, v_j\}$, with $j > i$ and ``clean-up" pairs $\{a_\ell, b_{\ell'}\}$. Proposer will propose all pairs of the first form $\{u_i, v_j\}$, but only the open pairs among $\{v_i, v_j\}$, and orders the choices completely at random. Let us fix some strategy for Decider; our goal is to show that no matter what this strategy is, Proposer will win w.h.p. We define $p_i$ to be the fraction (out of $n/3$) of Decider's answers which are YES during period $i$; $p_i$ will depend on Decider's strategy, and potentially also on the random permutation chosen by Proposer, or more generally on the state of the game throughout period $i$.

Recall also that we organized the periods into {\it epochs} $I_k\subseteq\{1,\ldots, n/6\}$ of decreasing order $|I_k| = \varepsilon_k n$ satisfying $\sum \varepsilon_k = 1/6$. We will pick $\varepsilon_k$ to satisfy $\varepsilon_k \ge 1/\log n$ if $n$ is sufficiently large. The vertices of an epoch we called $U_k=\{u_i\}_{i\in I_k}$ and $V_k = \{v_i\}_{i\in I_k}$. We define $P_k$ to be the average of $p_i$ over all $i\in I_k$. Finally, we define $o_k$ to be the \textit{open density} inside $V_k$ at the end of epoch $k-1$; formally, if $i^*$ is the maximum $i \in I_{k-1}$, then we define
\[
    o_k=\frac{|O_{i^*} \cap \binom{V_k}2|}{\binom{|V_k|}2}.
\]

The main result that we need is the following theorem of Shearer\footnote{Shearer's bound is of the form $\alpha(G) \geq n f(d)$ with $f(d) =(1+o(1))\log d/d$, and one can check that $f(d) \geq 2\log d/(3 d)$ for all $d$.}~\cite{Sh1}, improving by a constant factor earlier work of Ajtai, Koml\'os, and Szemer\'edi~\cite{AjKoSz}. As usual, $\alpha(G)$ denotes the order of the largest independent set in $G$.
\begin{thm}\label{thm:aks}
	If $G$ is a triangle-free graph on $n$ vertices and with average degree $d$, then
	\[
		\alpha(G) \geq \frac{2n \log d}{ 3 d}.
	\]
\end{thm}

We first show, using Theorem \ref{thm:aks}, that if $P_k$ is too small for a given epoch, Proposer wins in this epoch immediately.

\begin{lem}\label{lem:pklarge-adaptive}
	Suppose that for some $k$, we have
	\[
		P_k \le  \frac{200}{ o_k \sqrt n}. 
	\]
	Then with probability at least $1-e^{-\Omega(\sqrt n/(\log n)^3)}$, Proposer will win the game; in fact, at the end of epoch $k$, $V_k$ will contain w.h.p.\ an independent set of order $\frac{1}{1000}\sqrt n \log n$, and since Proposer will have proposed or closed every pair in this independent set by the end of epoch $k$, it will remain independent until the end of the game. 
\end{lem}
\begin{proof}
First, suppose that at the end of epoch $k$, the average degree inside $V_k$ is at most $3$. By Tur\'an's Theorem, a graph on $\varepsilon_k n$ vertices with average degree $d$ has independence number at least $\varepsilon_k n/(d+1)$, so $V_k$ will contain an independent set of size at least
\[
    \frac{\varepsilon_k n}{4} \geq \frac{n}{4\log n} \geq \frac{1}{1000} \sqrt n \log n,
\]
for $n$ sufficiently large. Thus, from now on, we may assume that the average degree inside $V_k$ is more than $3$ at the end of epoch $k$.

Next, suppose that $o_k \leq 1/n$. This implies that at most $\binom{|V_k|}2/n$ edges can be built inside $V_k$, so the average degree in $V_k$ is at most $|V_k|/n=\varepsilon_k<3$. This contradicts the above assumption, so we may assume that $o_k>1/n$.

Next, we show that if $P_k < \varepsilon_k/\sqrt{n}$, then Proposer is guaranteed to win in epoch $k$. Indeed, the total number of edges built in the entire epoch is $P_k \cdot (n/3)\cdot (\varepsilon_k n) = P_k \varepsilon_k n^2/3$. In particular, at most this many edges are built inside $V_k$. It follows that the average degree in $V_k$ at the end of the game is between $3$ and $P_k n/3$, so by Theorem~\ref{thm:aks}, $V_k$ will contain an independent set of order at least
\[
\frac{2|V_k|\log(P_k n/3)}{3(P_k n/3)}=\frac{2\varepsilon_k \log (P_k n/3)}{P_k},
\]
where we use the fact that the function $\log d/d$ is monotonically decreasing for $d>3$. If $P_k < \varepsilon_k / \sqrt{n}$, then this quantity is at least $\frac {1}{1000} \sqrt{n} \log n$ for $n$ sufficiently large, and Proposer wins.

Now, using this fact we may assume that Decider chooses YES at least an $\varepsilon_k/\sqrt{n}$ fraction of the time throughout the period $k$. We can now apply Lemma~\ref{lem:bucket} to obtain a stronger bound on the total number of edges built in $V_k$. Namely, let $A$ be the set of all pairs within $V_k$ that are proposed in epoch $k$. Since the average open density at the beginning of the epoch is $o_k$ and only open pairs can be proposed, $|A|\le o_k \binom{|V_k|}2 \leq o_k \varepsilon_k^2 n^2/2$. If $A$ is smaller, then add pairs to $A$ until $|A|=o_k \varepsilon_k^2 n^2/2$. Also, a total number of $\varepsilon_k n^2/3$ turns occur in the epoch, and we know from the preceding argument that Decider answers YES at least $\varepsilon_k/\sqrt{n}$ of the time, so we can take $\nu_0 = \varepsilon_k ^2 n^{3/2}/3$ as a lower bound on the number of YES answers Decider gives.

Thus, we can think of epoch $k$ as an instance of the Coins and Buckets game with parameters
\[
(a,\nu, \nu_0) = \left(\frac{o_k \varepsilon_k^2 n^2}2, \frac{\varepsilon_k n^2}3, \frac{\varepsilon_k ^2 n^{3/2}}3\right).
\]
Here we think of the set $A$ of pairs as bucket $A$ in the Coins and Buckets game, and the condition $a < \nu/2$ is satisfied because $\varepsilon_k < \frac{1}{6}$ and $o_k \leq 1$.

Each pair in $A$ proposed corresponds to a coin placed in the bucket $A$, each pair proposed outside $A$ corresponds to a coin in the bucket $B$, and Decider's answers of heads/tails correspond to YES/NO.

We now apply Lemma~\ref{lem:bucket} with $t = a/2\nu=3o_k \varepsilon_k/4$ to this game. By our choice of $\nu_0$, we know that $h \geq \nu_0$ always, so we find that
\begin{align*}
    \pr \left(h_A \geq \left(\frac{a}{\nu}+t\right)h\right) &\leq \frac{80\sqrt \nu}{t^2} e^{-\frac{\nu_0 \nu t^2}{20a}}\\
    &=\frac{80\sqrt{\varepsilon_k n^2/3}}{(3o_k \varepsilon_k/4)^2} e^{-\frac{o_k \varepsilon_k^3 n^{3/2}}{160}}\\
    &\leq \frac{80 n^3 \sqrt{\varepsilon_k/3}}{9 \varepsilon_k^2/16} e^{-\varepsilon_k^3 \sqrt n/160}\\
    &\leq e^{-\Omega(\sqrt n/(\log n)^3)},
\end{align*}
where the second inequality follows from our assumption that $o_k \geq 1/n$, and the final inequality uses the fact that $\varepsilon_k \geq 1/\log n$ for $n$ sufficiently large. 

Therefore, we find that the number of edges actually in $V_k$ is bounded above by
\[
h_{A} \leq \left(\frac a\nu +t\right)h= \left(1 + \frac{1}{2}\right) P_k o_k \varepsilon_k ^2 n^2/2,
\]
with probability at least $1-e^{-\Omega(\sqrt n/(\log n)^3)}$. 

Thus the average degree in $V_k$ at the end of the game will be at most $\frac 32 P_ko_k\varepsilon_k n$. By Theorem~\ref{thm:aks} again, $V_k$ must contain an independent set of order at least
\[
    \frac{|V_k| \log (\frac 32 P_k o_k \varepsilon_k n)}{\frac 94 P_k o_k \varepsilon_k n} \geq \frac{(\varepsilon_k n) \log (\frac 32 (200/\sqrt n) \varepsilon_k n)}{\frac 94 (200/\sqrt n) \varepsilon_k n}\geq \frac{\sqrt n}{450} \log (300 \sqrt n/\log n) \geq \frac{1}{1000} \sqrt n \log n,
\]
where we used the bounds $P_k o_k \leq 200/\sqrt n$ and $\varepsilon_k \geq 1/\log n$. Thus, in this regime of $P_k$, Proposer wins with probability at least $1-e^{-\Omega(\sqrt n/(\log n)^3)}$.
\end{proof}

This lemma implies that for Decider to have a reasonable hope of winning, Decider must always ensure that $P_k \geq 200/(o_k \sqrt n)$. Intuitively, a large $P_k$ implies that many edges will be built between $U_k$ and $V_{k+1}$, which means that many pairs in $V_{k+1}$ will be closed off, so $o_{k+1}$ must be considerably smaller than $o_k$. Thus, to ensure $P_{k+1} \geq 200/(o_{k+1}\sqrt n)$, $P_{k+1}$ must be considerably larger than $P_k$, and so on. This cannot be sustained for long because $P_k$ must always be bounded by $1$, so eventually Decider will run out of room and lose the game. The next two lemmas makes this rigorous. We will need further notation; define $o_{k,i}$ to be the fraction of pairs in $\binom{V_k}2$ that are open at the end of period $i$.

\begin{lem} \label{lem:open-shrinking}
Suppose that $k\ge 2$, $i\in I_{k-1}$, and $p_i, o_{k,i-1}$ satisfy $p_i \ge P_{k-1}/2$ and
\begin{equation}\label{eq:density-needed}
o_{k,i-1} \ge \frac{3(\log n)^3}{P_{k-1} \varepsilon_k n},
\end{equation}
then with probability at least $1-e^{-\Omega(\log n)^2}$,
\[
o_{k,i} \le \left(1-\frac{1}{16}p_i ^2\right)o_{k,i-1}.
\]
\end{lem}
\begin{proof}
We set up the conditions for Coins and Buckets to apply Lemma \ref{lem:bucket}. Suppose $L_i$ is the set of all pairs Proposer proposes during period $i$, consisting of all the pairs $\{u_i,v_j\}$ with $j>i$, all the open pairs $\{v_i, v_j\}$ with $j>i$, and some filler pairs $\{a_\ell,b_{\ell'}\}$ to make at total of $n/3$ turns. Suppose $A\subseteq L_i$ has order $(\log n)^2/P_{k-1} \leq |A|\le |L_i|/2$. The total number of YES responses given in this period is $p_i |L_i| \ge P_{k-1}|L_i|/2$, so we can define $\nu_0 = P_{k-1}|L_i|/2$. Thus we can think of this period as a game of Coins and Buckets with parameters $(a, \nu, \nu_0) = (|A|, |L_i|, P_{k-1}|L_i|/2)$.

Then, by Lemma~\ref{lem:bucket}, the density of YES responses made by Decider to pairs in $A$ will be close to the density of YES responses made by Decider to all pairs in $A_i$. Specifically, taking
\[
t = \frac{a}{2\nu}=\frac{3|A|}{2n}
\]
and writing $y_A$ for the number of YES answers given to edges in $A$, we have
\begin{equation}\label{eq:bucket}
y_A = \left(1\pm \frac 12\right) p_i |A|,
\end{equation}
with probability at least 
\begin{align*}
    1- \frac{80\sqrt \nu}{t^2}e^{-\frac{\nu_0 \nu t^2}{20a}}&=1-\frac{80\sqrt{n/3}}{(3|A|/2n)^2} e ^{-\frac{(P_{k-1}n/6)|A|}{80n/3}}\\
    &\geq 1- 100 n^{5/2} e^{-(\log n)^2/160}\\
    &=1-e^{-\Omega(\log n)^2},
\end{align*}
by Lemma \ref{lem:bucket}. In the inequality, we use the trivial bound $|A| \geq 1$ and our assumption that $|A| \geq (\log n)^2/P_{k-1}$. We will apply bound (\ref{eq:bucket}) a total of $O(n)$ times, so by the union bound, all of them will hold with probability at least $1-e^{-\Omega(\log n)^2}$.

Now, we count the number of open pairs in $V_k$ before and after period $i$. Thinking of the set $O_i$ of open pairs after turn $i$ as itself a graph on $V$, let $O_{k,i} = O_i[V_k]$ be the induced subgraph of $O_i$ consisting of open pairs in $V_k$ after period $i$. The pairs of $O_{k,i-1}$ that are closed off after period $i$ are exactly those pairs $\{v_j, v_{j'}\}$ for which Decider answers YES to both $\{u_i, v_j \}$ and $\{u_i, v_{j'}\}$, or else to both $\{v_i,v_j\}$ and $\{v_i,v_{j'}\}$. In particular, to lower-bound the number of pairs that are closed off in period $i$, it suffices to lower-bound the number of $j,j'$ for which Decider answers YES to both $\{u_i,v_j\}$ and $\{u_i,v_{j'}\}$. 

We write $N_{O_{k,i-1}}(v_j)$ for the neighborhood of $v_j$ in $O_{k,i-1}$, and $d_{k,i-1}(v_j)$ for the order of this neighborhood. We will apply inequality (\ref{eq:bucket}) to the set $A(v_j)=\{u_i\} \times N_{O_{k,i-1}}(v_j)$ consisting of all pairs from $u_i$ to $N_{O_{k,i-1}}(v_j)$. We can only apply it to those $v_j$ with
\[
    |A(v_j)|=d_{{k,i-1}}(v_j) \geq \frac{(\log n)^2}{P_{k-1}};
\]
note that the inequality $|A(v_j)| \leq |L_i|/2$ holds for all $v_j$, since at most half the pairs in $L_i$ are between $u_i$ and $V_k$, and in particular at most half are in $A(v_j)$. Then for those $v_j$ with $d_{{k,i-1}}(v_j) \geq (\log n)^2/P_{k-1}$, inequality (\ref{eq:bucket}) tells us that with probability at least $1-e^{-\Omega(\log n)^2}$, the number of neighbors $j' \in N_{O_{k,i}}(v_j)$ for which Decider answers YES to $\{u_i, v_{j'}\}$ is $(1 \pm \frac 12)p_i d_{k,i-1}(v_j)$. We will write $y_{N_{O_{k,i-1}}(v_j)}$ for this number of YES answers, which is a shorthand for $y_{\{u_i\} \times N_{O_{k,i-1}}(v_j)}$. 

We also divide the vertices $v_j \in V_k$ into dyadic intervals based on degree. That is, let $D_{\ell}$ be the set of all vertices $v_j\in V_k$ with $d_{k,i-1}(v_j)$ lying in $[2^{\ell-1}, 2^{\ell})$, where $1\le \ell \le \log_2 n$. Applying inequality (\ref{eq:bucket}) again, to the set $\{u_i\} \times D_\ell$, we have that if $|D_\ell| \geq (\log n)^2/P_{k-1}$, then $y_{D_\ell} = (1 \pm \frac 12) p_i |D_\ell|$ with probability at least $1-e^{-\Omega(\log n)^2}$. Here we again write $y_{D_\ell}$ as a shorthand for $y_{\{u_i\} \times D_\ell}$, the number of YES answers made to pairs of the form $\{u_i, v_j\}$ where $v_j \in D_\ell$. 

Throughout, we applied inequality (\ref{eq:bucket}) a total of $O(n)$ times, so everything still holds with probability $1-e^{-\Omega(\log n)^2}$. Putting all this together, we can lower bound the number of open pairs that are closed off in period $i$, by summing over $\ell$. For each $\ell \geq 1+\log ((\log n)^2/P_{k-1})$, we have that every $v_j \in D_\ell$ has degree at least $2^{\ell-1} \geq (\log n)^2/P_{k-1}$, so we can apply our above lower bound for $y_{N_{O_{k,i-1}}(v_j)}$. Suppose additionally that $\ell$ is such that $|D_\ell| \geq (\log n)^2/P_{k-1}$, so that we also have a bound for $y_{D_\ell}$. Then we find that for such an $\ell$, the number of pairs $\{v_j, v_{j'}\}$ closed off with $v_j \in D_\ell$ is at least
\[
y_{D_\ell} \cdot \min_{v_j \in D_\ell} y_{N_{O_{k,i-1}}(v_j)} \ge \left( \frac 12 p_i |D_\ell|\right) \left( \frac 12 p_i \min_{v_j \in D_\ell} d_{k,i-1}(v_j)\right) \geq \frac 14 p_i^2 |D_\ell| 2^{\ell-1}.
\]
We now sum this up over all $\ell$ as above, namely all $\ell \geq 1+\log ((\log n)^2/P_{k-1})$ with $|D_\ell| \geq (\log n)^2/P_{k-1}$; call such values of $\ell$ \emph{good}. Doing so gives us a lower bound on the total number of open pairs in $V_k$ closed off during period $i$; note that when we sum up, we might double-count pairs $\{v_j,v_{j'}\}$ if $v_j$ and $v_{j'}$ are in different parts of the dyadic partition $\{D_\ell\}$. Thus, we need to divide by $2$ when summing, and find that the total number of open pairs in $V_k$ that are closed during period $i$ is at least
\begin{align}
    \frac{p_i^2}{8} \sum_{\ell \text{ good}} |D_\ell| 2^{\ell-1}&=\frac{p_i^2}{16} \sum_{\ell \text{ good}} |D_\ell| 2^\ell \nonumber\\
    &\geq \frac{p_i^2}{16} \sum_{\ell \text{ good}} \sum_{v_j \in D_\ell} d_{k,i-1}(v_j) \nonumber\\
    &=\frac{p_i^2}{16} \left( 2|O_{k,i-1}|-\sum_{\ell\text{ not good}} \sum_{v_j \in D_\ell} d_{k,i-1}(v_j)\right). \label{eq:sumoverdl}
\end{align}
So it suffices to bound $\sum_{v_j \in D_\ell} d_{k,i-1}(v_j)$ for all $\ell$ that are not good. If $\ell \leq \log((\log n)^2/P_{k-1})$, then
\[
    \sum_{v_j \in D_\ell} d_{k,i-1} (v_j) \leq |D_\ell| 2^\ell \leq |D_\ell| \frac{(\log n)^2}{P_{k-1}} \leq \frac{(\log n)^2}{P_{k-1}} \cdot \varepsilon_k n.
\]
On the other hand, if $\ell$ is not good since $|D_\ell| <(\log n)^2/P_{k-1}$, then we find that
\[
    \sum_{v_j \in D_\ell} d_{k,i-1}(v_j) \leq |D_\ell| \max_{v_j \in D_\ell} d_{k,i-1}(v_j) \leq \frac{(\log n)^2}{P_{k-1}} \cdot \varepsilon_k n.
\]
Since there are are at most $\log n$ values of $\ell$, and in particular at most $\log n$ values of $\ell$ that are not good, we find that 
\[
    \sum_{\ell\text{ not good}} \sum_{v_j \in D_\ell} d_{k,i-1}(v_j) \leq \frac{(\log n)^3}{P_{k-1}} \cdot \varepsilon_k n.
\]
We can now plug this back into (\ref{eq:sumoverdl}). Doing so, we find that the number of open pairs in $V_k$ closed off during period $i$ is at least
\[
    |O_{k,i-1}|-|O_{k,i}| \geq \frac{p_i^2}{16} \left( 2 |O_{k,i-1}| - \frac{\varepsilon_k n (\log n)^3}{P_{k-1}}\right).
\]
We assumed that $o_{k,i-1} \geq 3(\log n)^3/P_{k-1} \varepsilon_k n$, which implies that
\[
    |O_{k,i-1}|=o_{k,i-1} \binom{|V_k|}2 \geq \frac{o_{k,i-1}|V_k|^2}{3} \geq \frac{\varepsilon_k n(\log n)^3}{P_{k-1}}.
\]
Putting this together, we find that
\[
    |O_{k,i-1}|-|O_{k,i}| \geq \frac{p_i^2}{16} |O_{k,i-1}|.
\]
Dividing out by $\binom{|V_k|}2$ gives the desired result for $o_{k,i}$.
\end{proof}

Now we are ready to accumulate the density decrements from this lemma, to show that $o_{k}$ is small in terms of $P_{k-1}$, as before. Recall that $o_k$ is the open density in $V_k$ just after the epoch $k-1$.

\begin{lem}\label{lem:ck-small-adaptive}
With probability at least $1-e^{-\Omega(\log n)^2}$, for all $k\ge 2$,
\[
o_k \le \max\left(\frac{3(\log n)^3}{\varepsilon_k P_{k-1} n}, e^{-\varepsilon_{k-1} P_{k-1}^2 n / 64}\right).
\]
\end{lem}
\begin{proof}
We will apply Lemma~\ref{lem:open-shrinking} using the fact that $P_{k-1}$ is the average value of $p_i$ across all $i$ in epoch $k-1$. If at any point $o_{k,i} \le \frac{3(\log n)^3}{\varepsilon_k P_{k-1} n}$, then we are done, since the open density can never increase. Also, we will ignore every period $i$ where $p_i < P_{k-1}/2$, and let $I_{k-1} ^*$ be the indices of the remaining periods.

The conditions of Lemma~\ref{lem:open-shrinking} are satisfied for each $i\in I_{k-1}^*$, so we have for these $i$,
\[
o_{k,i} \le \left(1-\frac{1}{16}p_i^2\right)o_{k,i-1}.
\]
Iterating over all $i\in I_{k-1}^*$ and noting that the open density is initially at most $1$, it follows that
\begin{equation}\label{eq:cumulative-open}
o_k \le \prod_{i\in I_{k-1}^*} \left(1-\frac{1}{16}p_i^2\right) \le \prod_{i\in I_{k-1}^*}\exp(-p_i^2/16) = \exp\left(-\sum_{i \in I_{k-1}^*}p_i^2 / 16\right).
\end{equation}

There are $\varepsilon_{k-1} n$ periods in epoch $k-1$, so the total sum of $p_i$ over $i \in I_{k-1}$ is $P_{k-1} \cdot \varepsilon_{k-1} n$, while the sum of those $p_i$ for which $p_i < P_{k-1}/2$ is at most half this amount. Thus, together with the Cauchy--Schwarz inequality,
\[
\sum_{i \in I_{k-1}^*}p_i^2 \ge \frac{1}{\varepsilon_{k-1} n} \cdot \left(\sum_{i \in I_{k-1}^*}p_i\right)^2 \ge \frac 14 P_{k-1}^2 \cdot \varepsilon_{k-1}n.
\]
Plugging into (\ref{eq:cumulative-open}), this implies that $o_k \leq e^{-\varepsilon_{k-1} P_{k-1}^2 n/64}$, as desired.
\end{proof}

\noindent From Lemmas~\ref{lem:pklarge-adaptive} and~\ref{lem:ck-small-adaptive}, we can complete the proof of Theorem~\ref{thm:Proposer-win-adaptive}.

\begin{proof}[Proof of Theorem~\ref{thm:Proposer-win-adaptive}]
We pick $m = \log ^* (n) + 1$, where $\log ^* (n)$ is the binary iterated logarithm function. We then pick $\varepsilon_k=1/(6 \cdot 2^k)$ for $1 \leq k \leq m-1$ and $\varepsilon_m=\varepsilon_{m-1}$, so that $\sum \varepsilon_k=1/6$. Moreover, since $\varepsilon_k \geq \varepsilon_m \geq 2^{-\log^*(n)}/6$, we get that $\varepsilon_k \geq 1/\log n$ for all $n$ sufficiently large, as we required. 

For each $k$, we may assume by Lemma~\ref{lem:pklarge-adaptive} that $P_{k-1} > 200/(o_{k-1} \sqrt{n})$, for otherwise Proposer will win the game w.h.p. Then with the above choices of $\varepsilon_k$ and this lower bound on $P_{k-1}$, Lemma~\ref{lem:ck-small-adaptive} implies a recursive bound on the open densities of the form
\begin{equation}\label{eq:open-recursion}
o_k \le \max\left( \frac{ 2^{k}(\log n)^3}{10\sqrt{n}} o_{k-1}, e^{-100/(2^k o_{k-1}^2)}\right),
\end{equation}

Suppose that for some $k\le m-1$, the first term in the maximum in (\ref{eq:open-recursion}) is larger. Then $o_k \le 2^{k}(\log n)^3 /10 \sqrt{n}$. Applying (\ref{eq:open-recursion}) again,
\[
o_{k+1} \le \max \left(\frac{2^{2k+1}(\log n)^6}{100 n}, \exp \left(-\frac{10000n}{2^{3k+1}  (\log n)^6}\right)\right) \le \frac{200}{\sqrt{n}}
\]
for $n$ sufficiently large, since $2^{3k} = O(\log n)$ for $k = O(\log^* n)$. It then follows by Lemma~\ref{lem:pklarge-adaptive} that Proposer can win in epoch $k+1$ w.h.p.

Thus, we may assume that the second expression on the right hand side of (\ref{eq:open-recursion}) is larger for every $k\le m-1$. Writing $x_k = -\log o_k$, the bound becomes
\[
x_k \ge \frac{100}{2^k} e^{2x_{k-1}},
\]
starting from $x_1 = 0$. It is easy to prove by induction that $x_k \ge t_k(2)$ for all $k\ge 2$, where $t_k(2)$ is a tower of $2$s of height $k$. Thus, $x_{m-1} \ge t_{\log^* n}(2) \ge n$. But then, $o_{m-1} = e^{-x_{m-1}}\le e^{-n}$, which implies that $o_{m-1} = 0$ since it must be a nonnegative rational with denominator less than $n^2$. Proposer wins by epoch $m-1$ in this case.

We have shown that Proposer guarantees w.h.p.\ the existence of an independent set of size $\frac{1}{1000} \sqrt{n} \log n$, as desired.
\end{proof}
\section{Concluding Remarks}
A natural open problem is to close the constant-factor gap in Theorem \ref{thm:main}. If we define $RPS(n)$ to be the largest $s$ for which Proposer can win with probability at least $1/2$ regardless of how Decider chooses to play, then Theorem \ref{thm:main} implies that $RPS(n)=\Theta(\sqrt n \log n)$.  We have made no attempt to optimize the constants $10^{-3}$ and $10^3$ that our proof provides, but even if we did, it is unlikely that they would match, meaning that we still don't fully understand the full asymptotic behavior of $RPS(n)$. In the language of thresholds, Theorem \ref{thm:main} asserts that $RPS(n)$ exhibits a coarse threshold; we conjecture that there is a sharp threshold.
\begin{conj}
    There is some constant $C>0$ so that 
    \[
        RPS(n)=(C+o(1))\sqrt n \log n.
    \]
    Further, for every $\varepsilon>0$, Proposer can win w.h.p.\ if $s < (C-\varepsilon)\sqrt n \log n$, while Decider can win w.h.p.\ if $s>(C+\varepsilon)\sqrt n \log n$ w.h.p.
\end{conj}
\noindent Moreover, we suspect that Decider's strategy of playing randomly is optimal.

Finally, there is a natural extension of Ramsey, Paper, Scissors to the case where we forbid a subgraph other than the triangle. Specifically, for any fixed graph $H$, we define the $H$-RPS game to be exactly as before, except that Proposer can never propose a pair that, if added, would form a copy of $H$. As before, Proposer wins if the independence number is at least $s$ when Proposer has no legal moves remaining. From this, we can define the quantity $RPS(H;n)$ to be the maximum $s$ for which Proposer can win with probability at least $\frac 12$. Thus, the above discussion concerns the special case $RPS(K_3;n)$, and it would be natural to search for analogues of Theorem \ref{thm:main} for $RPS(H;n)$, where $H$ is some graph other than $K_3$. A natural question is whether Decider's random strategy is still close to optimal.

\end{document}